\documentclass[a4paper,11pt,reqno]{amsart}

\usepackage[T1]{fontenc}
\usepackage{microtype}
\usepackage{silence}
\usepackage{amsmath,amssymb,mathtools}
\usepackage{mathrsfs}
\usepackage{bm}
\usepackage{esint}

\usepackage{graphicx}
\usepackage{booktabs}

\usepackage{enumitem}

\usepackage[dvipsnames]{xcolor}
\definecolor{refblue}{RGB}{0,0,255}

\usepackage[
  colorlinks=true,
  linkcolor=refblue,
  citecolor=refblue,
  urlcolor=refblue,
  pdfborder={0 0 0}
]{hyperref}

\usepackage[nameinlink,noabbrev]{cleveref}

\numberwithin{equation}{section}

\theoremstyle{plain}
\newtheorem{theorem}{Theorem}[section]
\newtheorem{lemma}[theorem]{Lemma}
\newtheorem{proposition}[theorem]{Proposition}
\newtheorem{corollary}[theorem]{Corollary}

\newtheorem{remark}[theorem]{Remark}

\theoremstyle{definition}

\crefname{theorem}{Theorem}{Theorems}
\crefname{lemma}{Lemma}{Lemmas}
\crefname{proposition}{Proposition}{Propositions}
\crefname{corollary}{Corollary}{Corollaries}
\crefname{claim}{Claim}{Claims}
\crefname{definition}{Definition}{Definitions}
\crefname{assumption}{Assumption}{Assumptions}
\crefname{example}{Example}{Examples}
\crefname{remark}{Remark}{Remarks}
\crefname{equation}{equation}{equations}
\crefname{section}{Section}{Sections}

\newcommand{\R}{\mathbb{R}}

\newcommand{\cE}{\mathcal{E}}

\DeclareMathOperator{\supp}{supp}

\title[Fract--Log Faber--Krahn Inequalities]
{On Faber-Krahn inequality for Dirichlet Fractional-Logarithmic Laplacian }

\author{Huyuan Chen}
\address{Center for Mathematics and Interdisciplinary Sciences,
Fudan University, Shanghai 200433, China;
Shanghai Institute for Mathematics and Interdisciplinary Sciences,
Shanghai 200433, China}
\email{chenhuyuan@yeah.net; chenhuyuan@simis.cn}

\author{Rui Chen}
\address{School of Mathematical Sciences, Fudan University,
Shanghai 200433, China;
Brandenburg University of Technology Cottbus--Senftenberg,
03046 Cottbus, Germany}
\email{chenrui23@m.fudan.edu.cn}

\author{Bobo Hua}
\address{School of Mathematical Sciences, Fudan University,
Shanghai 200433, PR China;
Shanghai Center for Mathematical Sciences, Fudan University,
Shanghai 200433, PR China}
\email{bobohua@fudan.edu.cn}


\begin{document}

\begin{abstract}
We study the Faber--Krahn problem for the fractional--logarithmic Laplacian
$(-\Delta)^{s+\ln}$, $0<s<1$, with Fourier symbol
$|\xi|^{2s}\ln|\xi|^2$. We prove a small-volume and large-volume dichotomy.
For sufficiently small volume, balls minimize the first Dirichlet eigenvalue,
with rigidity up to translation. As a consequence, the Faber--Krahn
inequality for the fractional Laplacian is recovered from the small-scale
limit.

For large volumes, after shifting the spectral minimum $-1/(es)$, we prove
\[
\lambda_1^{s+\ln}(\Omega)+\frac1{es}
\gtrsim
|\Omega|^{-\frac{4}{n+1}}.
\]
This rate is attained by anisotropic boxes, while for balls
\[
\lambda_1^{s+\ln}(B)+\frac1{es}
\asymp
|B|^{-\frac2n}.
\]
Hence, for $n\geq2$, sufficiently large anisotropic boxes have strictly
smaller first eigenvalue than balls of the same volume, and the
Faber--Krahn inequality fails in the large-volume regime. The exponent
$4/(n+1)$ reflects the codimension-one minimizing sphere and the second
and fourth order growth of the symbol in the normal and tangential
directions, respectively. We also prove the rigidity of the
Faber--Krahn inequality for the logarithmic Laplacian, resolving the open problem posed in
\cite{ChenWeth2019}.
\end{abstract}

\maketitle

\section{Introduction and Main Results}

In this paper, we study the Faber--Krahn inequality for the
fractional-logarithmic Laplacian
\[
(-\Delta)^{s+\ln}
:=
\left.
\frac{d}{dt}(-\Delta)^t
\right|_{t=s},
\qquad
0<s<1,
\]
where $(-\Delta)^s$ denotes the fractional Laplacian
\[ (-\Delta)^s u(x)
=
c_{n,s}\,\mathrm{p.v.}\int_{\mathbb R^n}\frac{u(x)-u(y)}{|x-y|^{n+2s}}dy,\qquad  c_{n,s}
:=
2^{2s}\pi^{-n/2}s\,
\frac{\Gamma\!\left(\frac{n+2s}{2}\right)}{\Gamma(1-s)}.\]
Its Fourier symbol is $|\xi|^{2s}\ln|\xi|^2$
and, for every $u\in C_c^2(\R^n)$, it admits the integral
representation
\[
(-\Delta)^{s+\ln}u(x)
=
\operatorname{p.v.}
\int_{\R^n}
\bigl(u(x)-u(y)\bigr)
K_{s+\ln}(x-y)\,dy,
\]
where
\[
K_{s+\ln}(z)
=
c_{n,s}
\bigl(b_{n,s}-2\ln|z|\bigr)
|z|^{-n-2s},\qquad b_{n,s}
=
\frac{d}{ds}\ln c_{n,s}.
\]

The operator was introduced in \cite{ChenChenHauer2026} as the derivative of
the fractional Laplacian with respect to its order. Subsequent works have
developed its potential theory and global regularity
\cite{Che26}, introduced the conformal fractional-logarithmic Laplacian on
the sphere together with related sharp inequalities \cite{CCH26}, and
studied small-order limits for fractional harmonic functions
\cite{JarohsSenWeth2026}. Motivated by these developments, in the present
paper we consider the first Dirichlet eigenvalue
$\lambda_1^{s+\ln}(\Omega)$ on bounded open sets
$\Omega\subset\R^n$ and study the corresponding Faber--Krahn problem:
for a prescribed volume, does the ball minimize
$\lambda_1^{s+\ln}(\Omega)$?

\smallskip

The Faber--Krahn inequality is classical for both the Laplacian and the
fractional Laplacian. For a bounded open set $\Omega\subset\R^n$, let $\lambda_1(\Omega)$
denote the first Dirichlet eigenvalue of $-\Delta$. Faber
\cite{Faber1923} and Krahn \cite{Krahn1925} proved that, for every ball
$B$ satisfying $|B|=|\Omega|$,
\[
\lambda_1(\Omega)\geq\lambda_1(B),
\]
with equality precisely for balls, up to translation. Equivalently,
\[
|\Omega|^{2/n}\lambda_1(\Omega)
\geq
|B|^{2/n}\lambda_1(B).
\]
Its sharp quantitative form was established by Brasco, De Philippis, and
Velichkov \cite{BrascoDePhilippisVelichkov2015}.

For $0<s<1$, let $\lambda_{1,s}(\Omega)$
be the first Dirichlet eigenvalue of $(-\Delta)^s$ with zero exterior
condition. The fractional Faber--Krahn inequality states that $\lambda_{1,s}(\Omega)\geq\lambda_{1,s}(B),$
or equivalently,
\[
|\Omega|^{2s/n}\lambda_{1,s}(\Omega)
\geq
|B|^{2s/n}\lambda_{1,s}(B),
\]
with equality only for balls, up to translation; see, for instance,
\cite[Theorem~3.5]{BrascoLindgrenParini2014}. Quantitative stability
estimates were obtained by Brasco, Cinti, and Vita
\cite{BrascoCintiVita2020}.

\smallskip

We find a different behavior for the fractional-logarithmic Laplacian:
the validity of the Faber--Krahn inequality depends on the volume. For
sufficiently small domains, the inequality holds and the ball is the unique
minimizer up to translation. In contrast, when $n\geq2$, the inequality
fails for sufficiently large volume, since suitably anisotropic boxes have
strictly smaller first eigenvalue than balls of the same volume. Thus the
fractional-logarithmic Laplacian exhibits a small-volume/large-volume
dichotomy which does not occur for the classical or fractional Laplacian.
We first state the small-volume result.

\begin{theorem}
\label{teo dom-1}
Let $n\geq1$ and $s\in(0,1)$. There exists
$V_0=V_0(n,s)>0$ such that, for every bounded open set
$\Omega\subset\R^n$ with $0<|\Omega|\leq V_0$ and every ball
$B\subset\R^n$ satisfying $|B|=|\Omega|$,
\[
\lambda_1^{s+\ln}(\Omega)
\geq
\lambda_1^{s+\ln}(B).
\]
If equality holds, then $\Omega$ coincides with a translate of $B$ up to a
null set. If, in addition, $\Omega$ has continuous boundary, then equality
holds if and only if $\Omega=B+x_0$
for some $x_0\in\R^n$.
\end{theorem}

The small-volume result can be understood from the sign structure of the
kernel. The positive part of $K_{s+\ln}$ is supported at short distances,
whereas the negative part appears only beyond the threshold $r_{n,s}$, see Section \ref{preli}.
After symmetric rearrangement, a nonnegative function supported on a set of
small measure is concentrated in a ball of diameter smaller than
$r_{n,s}$, so the negative interaction disappears. For sign-changing
functions, the interaction between the positive and negative parts remains,
but its contribution is of lower order in the volume.

To make this argument quantitative, let $B_R$ be the ball satisfying
$|B_R|=|\Omega|$. We first prove a rearrangement inequality for
nonnegative functions and then compare the first eigenvalues of balls with
different volumes by means of the scaling formula. For a general function
$u=u_+-u_-$, the interaction between $u_+$ and $u_-$ is controlled by the
negative part of the kernel, while the loss of volume of the two nodal
parts produces a positive term of order $R^{-2s}\log\frac1R.$
Since this term dominates the negative interaction when $R$ is sufficiently
small, the Faber--Krahn inequality follows. The equality case forces the
first eigenfunction to have one sign, and the rigidity is then obtained
from the equality case in the Riesz rearrangement inequality.

\smallskip

The situation changes in the large-volume regime. The Fourier symbol $|\xi|^{2s}\ln|\xi|^2$
attains its minimum $-\frac1{es}$
on the sphere
\[
|\xi|=\rho_s,
\qquad
\rho_s:=e^{-1/(2s)}.
\]
This is essentially different from the classical and fractional Laplacians,
whose symbols $|\xi|^2$ and $|\xi|^{2s}$ attain their minimum only at
$\xi=0$. Here the minimizing set is the codimension-one sphere
$\rho_s\mathbb S^{n-1}$. Near this sphere, the symbol grows quadratically in
the normal direction but only quartically along tangential directions. This
anisotropy leads to a different large-volume scale and allows suitably
elongated domains to have smaller first eigenvalue than balls of the same
volume.

\begin{theorem}
\label{teo dom-2}
Let $n\geq2$ and $s\in(0,1)$. The following assertions hold.

\medskip
\noindent
$(i)$ There exist constants $c_{n,s}>0$ and
$V_0=V_0(n,s)>0$ such that, for every bounded open set
$\Omega\subset\R^n$ with $|\Omega|\geq V_0$,
\[\lambda_1^{s+\ln}(\Omega)
\geq
-\frac1{es}
+
c_{n,s}|\Omega|^{-\frac{4}{n+1}}.\]

\medskip
\noindent
$(ii)$ For $l,r>0$, set
\[
Q_{l,r}
:=
(-l,l)\times(-r,r)^{n-1},
\]
and, for $V>0$, define
\[
Q_V
:=
Q_{V^{\frac{2}{n+1}},\,V^{\frac{1}{n+1}}}.
\]
Then there exists $V_1=V_1(n,s)\geq V_0$ such that, for every
$V\geq V_1$,
\[
\lambda_1^{s+\ln}(Q_V)
<
\lambda_1^{s+\ln}(B_V),
\]
where $B_V$ is any ball satisfying $|B_V|=|Q_V|=2^nV.$
\end{theorem}

The estimates used in the proof of Theorem~\ref{teo dom-2} give more precise information on the
large-volume behavior.

\begin{remark}
\medskip
\noindent
$(i)$ For the anisotropic boxes $Q_V$, there exists $C_{n,s}>0$ such that
\[
\lambda_1^{s+\ln}(Q_V)
+\frac1{es}
\leq
C_{n,s}V^{-\frac{4}{n+1}}
\]
for all sufficiently large $V$.

\medskip
\noindent
$(ii)$ Let $B_V$ be a ball satisfying $|B_V|=|Q_V|=2^nV.$
Then there exists $C_{n,s}\geq1$ such that
\[
C_{n,s}^{-1}V^{-\frac2n}
\leq
\lambda_1^{s+\ln}(B_V)+\frac1{es}
\leq
C_{n,s}V^{-\frac2n}
\]
for all sufficiently large $V$. 
Thus both $\lambda_1^{s+\ln}(Q_V)$ and
$\lambda_1^{s+\ln}(B_V)$ converge to the bottom
$-1/(es)$ of the whole-space Fourier symbol.
\end{remark}

The proof is based on the geometry of the minimum set of the shifted symbol
\[
\tau_s(\xi)
=
|\xi|^{2s}\ln|\xi|^2+\frac1{es}.
\]
Lemma~\ref{lem:symbol-geometry} shows that $\tau_s(\xi)
\asymp
\operatorname{dist}\bigl(\xi,\rho_s\mathbb S^{n-1}\bigr)^2$
near the minimizing sphere and, more precisely,
\[
\tau_s(\rho_se_1+\eta)
\lesssim
\eta_1^2+|\eta'|^4.
\]
Thus the symbol grows quadratically in the normal direction and quartically
in tangential directions. For a general domain of volume $V$, the
Tomas--Stein restriction theorem \cite[Section~2.2, equation~(2.7)]{CueninMerz2021}
(see also \cite[Theorem~3]{Stein1986} and \cite{Tomas1975}) controls the Fourier mass in a
$\delta$-neighborhood of $\rho_s\mathbb S^{n-1}$ by 
\[
\int_{\{||\xi|-\rho_s|<\delta\}}
|\widehat u(\xi)|^2\,d\xi
\lesssim
\delta V^{\frac{2}{n+1}}.
\]
Taking $\delta\asymp V^{-2/(n+1)}$ leaves a fixed amount of Fourier mass
outside this neighborhood, where $\tau_s\gtrsim\delta^2$. This gives $\mu_s(\Omega)
\gtrsim
V^{-\frac{4}{n+1}},$
which proves the lower bound in Theorem~\ref{teo dom-2}.

\smallskip

The anisotropic estimate suggests the boxes $Q_V$. We test the shifted
Rayleigh quotient with
\[
u_{l,r}(x)
=
l^{-1/2}r^{-(n-1)/2}
e^{i\rho_sx_1}
\phi\left(\frac{x_1}{l}\right)
\psi\left(\frac{x'}{r}\right).
\]
Its Fourier transform is concentrated near $\rho_se_1$ at scales
$l^{-1}$ in the normal direction and $r^{-1}$ in the tangential directions.
Lemma~\ref{lem:symbol-geometry} then gives $\mu_s(Q_{l,r})
\lesssim
l^{-2}+r^{-4}.$
Under the volume constraint $lr^{n-1}\asymp V,$
the two terms are balanced by $l\asymp r^2,\,
r\asymp V^{\frac1{n+1}},$
and hence $\mu_s(Q_V)
\lesssim
V^{-\frac{4}{n+1}}.$

\smallskip

For balls, the relevant scale is different. If $B_R$ is a large ball, an
isotropically localized test function modulated by $e^{i\rho_sx_1}$ gives $\mu_s(B_R)\lesssim R^{-2}.$
For the reverse estimate, localization of $u$ to $B_R$ yields $\|\nabla_\xi\widehat u\|_2\leq R.$
After cutting off the Fourier transform near $|\xi|=\rho_s$, the dilation
commutator
\[
i[T-\rho_s,A]=T,
\qquad
T:=\sqrt{-\Delta},
\]
gives $\|(T-\rho_s)u\|_2\gtrsim R^{-1}.$
Together with $\tau_s(\xi)\gtrsim\bigl(|\xi|-\rho_s\bigr)^2,$
this yields
\[
\mu_s(B_R)\gtrsim R^{-2}\asymp |B_R|^{-2/n}.
\]

Indeed, for $r>0$, scaling gives
\[
r^{2s}\lambda_1^{s+\ln}(r\Omega)
=
\inf_{\substack{
u\in\mathcal H_0^{s+\ln}(\Omega)\\
\|u\|_2=1}}
\left\{
\mathcal E_{s+\ln}(u,u)
+
2\ln\frac1r\,\mathcal E_s(u,u)
\right\}.
\]
Consequently, $\lim_{r\downarrow0}
\frac{r^{2s}\lambda_1^{s+\ln}(r\Omega)}
{2\ln(1/r)}
=
\lambda_{1,s}(\Omega).$
For $r>0$ sufficiently small, both $r\Omega$ and $rB$ fall into the
small-volume regime of Theorem~\ref{teo dom-1}. Hence
\[
\lambda_1^{s+\ln}(r\Omega)
\geq
\lambda_1^{s+\ln}(rB).
\]
Multiplying by $r^{2s}$, dividing by $2\ln(1/r)$, and letting
$r\downarrow0$ gives $\lambda_{1,s}(\Omega)
\geq
\lambda_{1,s}(B).$

\smallskip

Another related nonlocal operator is the logarithmic Laplacian
$(-\Delta)^{\ln}$, whose Fourier symbol is $2\ln|\xi|$
and which arises as the derivative of the fractional Laplacian at
$s=0$.
Chen and Weth \cite{ChenWeth2019} established the representation
\[
\begin{aligned}
(-\Delta)^{\ln}u(x)
={}&
c_n\,\operatorname{p.v.}
\int_{B_1(0)}
\frac{u(x)-u(x+y)}{|y|^n}\,dy-
c_n
\int_{\R^n\setminus B_1(0)}
\frac{u(x+y)}{|y|^n}\,dy
+
\rho_nu(x),
\end{aligned}
\]
where
\[
c_n=\frac{\Gamma(n/2)}{\pi^{n/2}},
\qquad
\rho_n=2\ln2+\psi(n/2)-\gamma,
\]
$\gamma$ denotes the Euler--Mascheroni constant and
$\psi=\Gamma'/\Gamma$ is the digamma function.

They also proved the Faber--Krahn inequality
\begin{equation}\label{FK-ln-1}
\lambda_1^{\ln}(\Omega)
\geq
\lambda_1^{\ln}(B),
\qquad
|B|=|\Omega|,
\end{equation}
for bounded Lipschitz domains. Their argument uses
\[
\lambda_1^{\ln}(\Omega)
=
\lim_{s\downarrow0}
\frac{\lambda_{1,s}(\Omega)-1}{s}
\]
together with the fractional Faber--Krahn inequality. As observed in
\cite{ChenWeth2019}, this limiting argument does not determine the equality
case, and the rigidity of \eqref{FK-ln-1} remained open. We give an
affirmative answer in the following theorem.

\begin{theorem}
\label{thm:rigidity}
Let $\Omega\subset\R^n$ be a bounded Lipschitz domain, and let $B$ be a ball
satisfying $|B|=|\Omega|$. Then $\lambda_1^{\ln}(\Omega)
\geq
\lambda_1^{\ln}(B).$
Moreover,
\[
\lambda_1^{\ln}(\Omega)
=
\lambda_1^{\ln}(B)
\]
if and only if $\Omega$ is a translate of $B$.
\end{theorem}

The proof uses the exact scaling law in Lemma \ref{lem:scaling1}
\[
\lambda_1^{\ln}(r\Omega)
=
\lambda_1^{\ln}(\Omega)-2\ln r,
\]
which preserves the difference of first eigenvalues under a common dilation.
We therefore scale the domains so that their diameter is smaller than one,
where the logarithmic Dirichlet form admits a strict rearrangement argument.
The equality case in the Riesz rearrangement inequality then yields the
rigidity.

\smallskip

  Faber--Krahn type problems have been studied for a variety of operators and
boundary conditions. For the Robin Laplacian with positive boundary
parameter, the ball remains the minimizer of the first eigenvalue; see, for
instance, \cite{DanersKennedy2007}. For mixed local--nonlocal operators,
Biagi, Dipierro, Valdinoci, and Vecchi
\cite{BiagiDipierroValdinociVecchi2023} proved a quantitative
Faber--Krahn inequality. Related spectral inequalities have also been
obtained for Dirac-type operators in suitable parameter regimes
\cite{BehrndtFrymarkHolzmannStelzer2024}, while the corresponding
shape-optimization problem on Riemannian manifolds was studied by
Lamboley and Sicbaldi \cite{LS20}.

On the other hand, the ball is not always optimal once the operator or the
boundary condition is changed. Freitas and Krej\v{c}i\v{r}\'ik
\cite{FreitasKrejcirik2015} disproved Bareket's conjecture for the Robin
Laplacian with sufficiently large negative boundary parameter, showing that
the ball need not maximize the first eigenvalue under a volume constraint;
see also the quantitative and small-coupling results
\cite{CitoLaManna2021,GavitoneKrejcirikPaoli2026}. For vector-valued
problems, Krej\v{c}i\v{r}\'ik, Lamberti, and Zaccaron \cite{KLZ25}
proved the failure of the Faber--Krahn inequality for the vector Laplacian,
and Henrot, Mazari-Fouquer, and Privat \cite{HMP24} showed that the ball is
not even a local minimizer for the three-dimensional Dirichlet--Stokes
problem. Reverse Faber--Krahn type inequalities and their limitations under
mixed Robin--Neumann boundary conditions were recently studied in
\cite{ABD25}.

The phenomenon obtained here is of a different form. We keep the same scalar
operator $(-\Delta)^{s+\ln}$, the same zero exterior Dirichlet condition,
and the same volume constraint, but the minimizing behavior changes with the
volume: balls minimize the first eigenvalue for sufficiently small volume,
whereas for sufficiently large volume they are outperformed by anisotropic
boxes. The change is caused by the nonhomogeneous Fourier symbol
$|\xi|^{2s}\ln|\xi|^2$ and, in particular, by its codimension-one set of
minimizers.

\smallskip
The remainder of the paper is organized as follows. Section~2 collects the
kernel decomposition and the basic estimates for the shifted Fourier symbol.
Section~3 studies the large-volume regime and proves the failure of the
Faber--Krahn inequality. Section~4 treats the small-volume Faber--Krahn
inequality and its equality case. The Appendix proves the rigidity result for
the logarithmic Laplacian.

\section{Kernel Decomposition and Fourier-Symbol Geometry}
\label{preli}

Throughout this section, let $n\geq1$ and $s\in(0,1)$. We use the unitary
Fourier transform
\[
\widehat u(\xi)
:=
(2\pi)^{-n/2}
\int_{\R^n}e^{-ix\cdot\xi}u(x)\,dx,
\qquad
\xi\in\R^n,
\]
and adopt the convention $r^{2s}\ln r^2=0$ at $r=0.$

We first recall the decomposition of the kernel of
$(-\Delta)^{s+\ln}$. For $r>0$, let
\[
K_{s+\ln}(r)
:=
c_{n,s}\bigl(b_{n,s}-2\ln r\bigr)r^{-n-2s},
\]
and set $r_{n,s}:=e^{b_{n,s}/2}.$
Define
\[
K_{s+\ln}
=
K_{s+\ln}^+-K_{s+\ln}^-,
\qquad
K_{s+\ln}^\pm
:=
\max\{\pm K_{s+\ln},0\}.
\]
Since
\[
b_{n,s}-2\ln r
\begin{cases}
>0,&0<r<r_{n,s},\\
=0,&r=r_{n,s},\\
<0,&r>r_{n,s},
\end{cases}
\]
we have
\[
K_{s+\ln}^-(r)=0
\quad\text{for }0<r\leq r_{n,s},
\qquad
K_{s+\ln}^+(r)=0
\quad\text{for }r\geq r_{n,s}.
\]
Moreover, $K_{s+\ln}^+$ is strictly decreasing on $(0,r_{n,s})$, since
\[
\frac{d}{dr}K_{s+\ln}^+(r)
=
c_{n,s}r^{-n-2s-1}
\Bigl[
-2-(n+2s)\bigl(b_{n,s}-2\ln r\bigr)
\Bigr]
<0.
\]
The negative part satisfies $K_{s+\ln}^-\in
L^1(\R^n)\cap L^\infty(\R^n).$

\smallskip

 For real-valued $u\in C_c^\infty(\Omega)$, the associated quadratic form is
\[
\mathcal E_{s+\ln}(u,u)
=
\int_{\R^n}
u(x)(-\Delta)^{s+\ln}u(x)\,dx.
\]
The first Dirichlet eigenvalue of $(-\Delta)^{s+\ln}$ is defined
variationally by
\[
\lambda_1^{s+\ln}(\Omega)
=
\inf_{\substack{
u\in\mathcal H_0^{s+\ln}(\Omega)\\
\|u\|_{L^2(\Omega)}=1}}
\mathcal E_{s+\ln}(u,u),
\]
where $\mathcal H_0^{s+\ln}(\Omega)
:=
\overline{C_c^\infty(\Omega)}
^{\,\|\cdot\|_{\mathcal H^{s+\ln}(\R^n)}}$ and $\mathcal H^{s+\ln}(\R^n)$ is given in \cite[Proposition 1.3]{ChenChenHauer2026}.

For $u\in\mathcal H_0^{s+\ln}(\Omega)$, the quadratic form is also given by
\[
\mathcal E_{s+\ln}(u,u)
=
\frac12
\iint_{\R^n\times\R^n}
(u(x)-u(y))^2
K_{s+\ln}(|x-y|)\,dx\,dy.
\]
Writing
\[
\mathcal I_\pm(u)
:=
\frac12
\iint_{\R^n\times\R^n}
(u(x)-u(y))^2
K_{s+\ln}^\pm(|x-y|)\,dx\,dy.
\]

\smallskip

We next record the corresponding properties of the Fourier symbol. Define
\begin{equation}\label{eq:shifted}
\tau_s(\xi)
:=
|\xi|^{2s}\ln|\xi|^2+\frac1{es},
\qquad
\xi\in\R^n.
\end{equation}
Since the function $m_s(r):=r^{2s}\ln r^2$
has its unique minimum at $\rho_s:=e^{-1/(2s)},$
with $m_s(\rho_s)=-\frac1{es},$
we have
\begin{equation}\label{eq:sphere}
    \tau_s(\xi)\geq0,
\qquad
\tau_s(\xi)=0
\quad\Longleftrightarrow\quad
|\xi|=\rho_s.
\end{equation}
Accordingly, for a bounded open set $\Omega\subset\R^n$, define the shifted
first eigenvalue
\begin{equation}\label{eq:shifted-rayleigh}
\mu_s(\Omega)
:=
\lambda_1^{s+\ln}(\Omega)+\frac1{es}
=
\inf_{\substack{
u\in\mathcal H_0^{s+\ln}(\Omega)\\
\|u\|_2=1}}
\int_{\R^n}
\tau_s(\xi)|\widehat u(\xi)|^2\,d\xi.
\end{equation}

The following estimates describe $\tau_s$ near and away from its minimizing
set.

\begin{lemma}
\label{lem:symbol-geometry}
There exist constants $c_s,C_s,\widetilde{c}_s,\widetilde{C}_s,\delta_s>0$ such that
\begin{equation}\label{eq:tau-quadratic}
c_s(r-\rho_s)^2
\leq
r^{2s}\ln r^2+\frac1{es}
\leq
C_s(r-\rho_s)^2
\end{equation}
for every $r\ge 0$ satisfying $|r-\rho_s|\leq \delta_s$. Moreover, for every
$0<\delta\leq\delta_s$,
\begin{equation}\label{eq:tau-away}
\inf_{\substack{r\ge0\\ |r-\rho_s|\geq\delta}}
\left(
r^{2s}\ln r^2+\frac1{es}
\right)
\geq
\widetilde{c}_s\delta^2.
\end{equation}
Finally, writing $\eta=(\eta_1,\eta')\in\R\times\R^{n-1}$ and
$e_1:=(1,0,\ldots,0)\in\R^n$, one has
\begin{equation}\label{eq:anisotropic-symbol}
\tau_s(\rho_s e_1+\eta)
\leq
C_s\bigl(\eta_1^2+|\eta'|^4\bigr),
\qquad
\eta\in\R^n,
\end{equation}
where $\tau_s$ and $\rho_s$ are defined in \eqref{eq:shifted} and \eqref{eq:sphere}, respectively.
\end{lemma}

\begin{proof}
Recall that $m_s(r)=r^{2s}\ln r^2$ and $\rho_s=e^{-1/(2s)}.$
Since
\[
m_s'(\rho_s)=0,
\qquad
m_s''(\rho_s)=4s\rho_s^{2s-2}>0,
\qquad
m_s(\rho_s)=-\frac1{es},
\]
the continuity of $m_s''$ and Taylor's theorem imply that there exist
$c_s,C_s,\delta_s>0$ such that
\[
c_s(r-\rho_s)^2
\leq
m_s(r)+\frac1{es}
\leq
C_s(r-\rho_s)^2
\]
whenever $|r-\rho_s|\leq\delta_s$. This proves
\eqref{eq:tau-quadratic}. To prove \eqref{eq:tau-away}, set
\[
m_0:=
\inf_{\substack{r\geq0\\ |r-\rho_s|\geq\delta_s}}
\left(m_s(r)+\frac1{es}\right).
\]
Since $m_s(r)+1/(es)$ vanishes only at $r=\rho_s$,
\[
m_s(r)+\frac1{es}\to\frac1{es}\quad\text{as }r\to0^+,
\qquad
m_s(r)+\frac1{es}\to+\infty\quad\text{as }r\to\infty,
\]
we have $m_0>0$. Hence, for every $0<\delta\leq\delta_s$ and
$|r-\rho_s|\geq\delta$, either $|r-\rho_s|\leq\delta_s$, in which case
\[
m_s(r)+\frac1{es}\geq c_s(r-\rho_s)^2\geq c_s\delta^2,
\]
or $|r-\rho_s|\geq\delta_s$, in which case $m_s(r)+\frac1{es}\geq m_0
\geq \frac{m_0}{\delta_s^2}\delta^2.$

\smallskip

It remains to prove \eqref{eq:anisotropic-symbol}. Writing
$\eta=(\eta_1,\eta')\in\R\times\R^{n-1}$, for $|\eta|$ sufficiently
small we have
\[
\begin{aligned}
\bigl||\rho_s e_1+\eta|-\rho_s\bigr|
&=
\frac{|2\rho_s\eta_1+|\eta|^2|}
     {|\rho_s e_1+\eta|+\rho_s} \leq
\widetilde{C}_s\bigl(|\eta_1|+\eta_1^2+|\eta'|^2\bigr)
\leq
\widetilde{C}_s\bigl(|\eta_1|+|\eta'|^2\bigr).
\end{aligned}
\]
Therefore, by \eqref{eq:tau-quadratic},
\[
\tau_s(\rho_s e_1+\eta)
\leq
\widetilde{C}_s\bigl(|\eta_1|+|\eta'|^2\bigr)^2
\leq
\widetilde{C}_s\bigl(\eta_1^2+|\eta'|^4\bigr).
\]
On compact subsets of $\R^n\setminus\{0\}$ the quotient $\frac{\tau_s(\rho_s e_1+\eta)}
     {\eta_1^2+|\eta'|^4}$
is bounded. Finally, as $|\eta|\to\infty$,
\[
\tau_s(\rho_s e_1+\eta)
\leq \widetilde{C}_s(\eta_1^2+|\eta'|^4),
\]
because $0<s<1$. Hence
\eqref{eq:anisotropic-symbol} holds.
\end{proof}

\begin{remark}
We give a geometric interpretation of Lemma~\ref{lem:symbol-geometry}.
For $n\geq2$,
\[
\{\xi\in\R^n:\tau_s(\xi)=0\}
=
\rho_s\mathbb S^{n-1},
\]
and \eqref{eq:tau-quadratic} implies that, near the minimizing sphere, $\tau_s(\xi)
\asymp
\operatorname{dist}\bigl(\xi,\rho_s\mathbb S^{n-1}\bigr)^2.$
At the point $\rho_se_1$, writing
\[
\eta=(\eta_1,\eta')\in\R\times\R^{n-1},
\]
the normal displacement $\eta_1$ changes the distance to the sphere at
first order, whereas the tangential displacement $\eta'$ changes it only at
second order. More precisely,
\[
\operatorname{dist}
\bigl(\rho_se_1+\eta,\rho_s\mathbb S^{n-1}\bigr)
\lesssim
|\eta_1|+|\eta'|^2,
\]
and hence $\tau_s(\rho_se_1+\eta)
\lesssim
\eta_1^2+|\eta'|^4.$
Thus the symbol has quadratic growth in the normal direction and quartic
growth in the tangential directions.

This anisotropic estimate is crucial for the large-volume analysis below.
It leads to the choice of longitudinal and transverse scales $l^{-2}\sim r^{-4},$
and, under the volume constraint $lr^{n-1}\sim V$, gives
\[
l\sim V^{\frac{2}{n+1}},
\qquad
r\sim V^{\frac{1}{n+1}},
\]
which yields the scale $V^{-\frac{4}{n+1}}$
for the first shifted eigenvalue.
\end{remark}

\section{Large-Volume Asymptotics of the First Eigenvalue}

In this section, we study the first Dirichlet eigenvalue of
$(-\Delta)^{s+\ln}$ in the large-volume regime. We first establish the
sharp volume scale
for general bounded domains and construct elongated boxes for which the same
order is attained. We then derive the two-sided estimate
for large balls. Finally, comparing these two decay rates shows that, when
$n\geq2$, sufficiently large elongated boxes have strictly smaller first
eigenvalue than balls of the same volume, and hence the Faber--Krahn
inequality fails in the large-volume regime.

\smallskip

We first recall  that  
$$
    \mu_s(\Omega)
:=
\lambda_1^{s+\ln}(\Omega)+\frac1{es}=\inf_{\substack{
u\in\mathcal H_0^{s+\ln}(\Omega)\\
\|u\|_{L^2(\Omega)}=1}}
\int_{\mathbb R^n}
\tau_s(\xi)|\widehat u(\xi)|^2\,d\xi
$$
and $\tau_s(\xi)
:=
|\xi|^{2s}\ln|\xi|^2+\frac1{es}\ge 0$ for all $\xi\in\mathbb R^n. $

\subsection{Sharp Large-Volume Estimates for the First Eigenvalue}

We first derive a uniform lower bound for the shifted first eigenvalue on
large domains. The main point is that the shifted symbol $\tau_s$ vanishes
quadratically on the minimizing sphere $|\xi|=\rho_s$. We use the
Tomas--Stein restriction estimate to control the $L^2$-mass of $\widehat u$
inside a thin neighborhood of this sphere in terms of the volume of
$\operatorname{supp}u$. Choosing the thickness of the shell optimally then
forces a fixed proportion of the Fourier mass to lie away from the minimizing
sphere, where $\tau_s$ is bounded below by a quadratic distance. This yields
the scale $V^{-4/(n+1)}$.

\begin{proposition}
\label{prop:volume-lower}
There exist constants $C_{n,s}^{(1)}>0$ and $V_0>0$ such that every bounded open subset $\Omega\subset\R^n$ with $|\Omega|=V\geq V_0$ satisfies
\[\mu_s(\Omega)\geq C_{n,s}^{(1)}V^{-4/(n+1)}.\]
\end{proposition}

\begin{proof}
Let $u\in\mathcal H_0^{s+\ln}(\Omega)$ satisfy $\|u\|_{L^2(\Omega)}=1$, and set $p_*:=\frac{2(n+1)}{n+3}.$ By H\"older's inequality, we have
\[
\|u\|_{L^{p_*}(\R^n)}
\leq
|\Omega|^{\frac1{p_*}-\frac12}\|u\|_{L^2(\R^n)}
=
V^{1/(n+1)}.
\]
For $0<\delta<\rho_s/2$, define the frequency shell
\[
\mathcal S_\delta
:=
\left\{
\xi\in\R^n:
\bigl||\xi|-\rho_s\bigr|<\delta
\right\}.
\]
By the Tomas--Stein restriction theorem
\cite[Section~2.2, equation~(2.7)]{CueninMerz2021}
(see also \cite[Theorem~3]{Stein1986} and \cite{Tomas1975}), if $p_*:=\frac{2(n+1)}{n+3},$
then, for every $f\in L^{p_*}(\R^n)$,
\begin{equation}\label{eq:TS-unit-sphere}
\left(
\int_{\mathbb S^{n-1}}
|\widehat f(\omega)|^2\,d\omega
\right)^{1/2}
\leq
C_n\|f\|_{L^{p_*}(\R^n)}.
\end{equation}
We apply this estimate to the rescaled function $f_r(x):=u(x/r),\,r>0.$
With our Fourier transform convention, by \eqref{eq:TS-unit-sphere},
\[
\begin{aligned}
\left(
\int_{\mathbb S^{n-1}}
|\widehat u(r\omega)|^2\,d\omega
\right)^{1/2}
&=
r^{-n}
\left(
\int_{\mathbb S^{n-1}}
|\widehat{f_r}(\omega)|^2\,d\omega
\right)^{1/2} \leq
C_n r^{-n}\|f_r\|_{L^{p_*}(\R^n)}.
\end{aligned}
\]
On the other hand, $\|f_r\|_{L^{p_*}(\R^n)}=r^{n/p_*}\|u\|_{L^{p_*}(\R^n)}.$
Consequently,
\[\left(
\int_{\mathbb S^{n-1}}
|\widehat u(r\omega)|^2\,d\omega
\right)^{1/2}
\leq
C_n
r^{-n+n/p_*}
\|u\|_{L^{p_*}(\R^n)}=C_n
r^{-\frac{n(n-1)}{2(n+1)}}
\|u\|_{L^{p_*}(\R^n)}..\]
Thus,
\[\left(
\int_{\mathbb S^{n-1}}
|\widehat u(r\omega)|^2\,d\omega
\right)^{1/2}
\leq
C_{n,s}\|u\|_{L^{p_*}(\R^n)},
\qquad
r\in
\left[\frac{\rho_s}{2},\frac{3\rho_s}{2}\right].\]
Hence, by polar
coordinates,
\[
\int_{\mathcal S_\delta}|\widehat u(\xi)|^2\,d\xi
\leq
C_{n,s}\delta\|u\|_{L^{p_*}(\R^n)}^2
\leq
C_{n,s}\delta V^{2/(n+1)}.
\]
Choose $\delta=c_0V^{-2/(n+1)},$
with $c_0>0$ sufficiently small so that the right-hand side is at most
$1/2$. For $V$ sufficiently large, $\delta\leq\delta_s$, and Plancherel's
identity yields $\int_{\mathcal S_\delta^c}|\widehat u(\xi)|^2\,d\xi\geq\frac12.$
By \eqref{eq:tau-away},
\[
\begin{aligned}
\int_{\R^n}\tau_s(\xi)|\widehat u(\xi)|^2\,d\xi
&\geq
\widetilde{c}_s\delta^2
\int_{\mathcal S_\delta^c}|\widehat u(\xi)|^2\,d\xi\geq
C_{n,s}^{(1)}V^{-4/(n+1)}.
\end{aligned}
\]
Taking the infimum in \eqref{eq:shifted-rayleigh} completes the proof.
\end{proof}

\begin{remark}
For the classical and fractional Dirichlet Laplacians, homogeneity and the
Faber--Krahn inequality give
\[
\inf_{|\Omega|=V}\lambda_1(\Omega)=C_nV^{-2/n},
\qquad
\inf_{|\Omega|=V}\lambda_{1,s}(\Omega)=C_{n,s}V^{-2s/n},
\]
with equality precisely for balls. In contrast, the symbol $|\xi|^{2s}\ln|\xi|^2$
of the fractional-logarithmic Laplacian is nonhomogeneous and attains its
minimum on the sphere $|\xi|=\rho_s$. Accordingly, after shifting the spectral
minimum, the large-volume scale is determined by the low-energy geometry near
this sphere, the exponent $4/(n+1)$ arises from the codimension-one minimizing set
and the Tomas--Stein restriction estimate.
\end{remark}

Recall that, for $l,r>0$ and $n\geq2$,
\[
Q_{l,r}
:=
(-l,l)\times(-r,r)^{n-1}
\subset\R\times\R^{n-1},
\]
and, for $V>0$, define
\begin{equation}\label{def q}
Q_V
:=
Q_{V^{\frac{2}{n+1}},\,V^{\frac{1}{n+1}}}.
\end{equation}

We next construct a family of elongated boxes for which the lower-bound
scale is attained. The test function is localized near the minimizing frequency
$\rho_s e_1$. We choose the longitudinal and transverse scales
$l$ and $r$ so that $l^{-2}\sim r^{-4},\,
lr^{n-1}\sim V,$
which gives $l\sim V^{2/(n+1)},\,
r\sim V^{1/(n+1)}.$

\begin{lemma}\label{prop:box-upper}
Let $n\geq2$ and let $Q_V$ be defined by \eqref{def q}. Then there exist
constants $C_{n,s}^{(2)}>0$ and $V_0>0$ such that
\[
\mu_s(Q_V)
\leq
C_{n,s}^{(2)}V^{-\frac{4}{n+1}}
\]
for every $V\geq V_0$.
\end{lemma}

\begin{proof}
Fix real-valued functions $\phi\in C_c^\infty(-1,1),\,
\psi\in C_c^\infty((-1,1)^{n-1}),$
such that $\|\phi\|_2=\|\psi\|_2=1.$
For $l,r>0$, define
\[
u_{l,r}(x)
:=
l^{-1/2}r^{-(n-1)/2}
e^{i\rho_sx_1}
\phi\left(\frac{x_1}{l}\right)
\psi\left(\frac{x'}{r}\right).
\]
Then $\|u_{l,r}\|_2=1,\,
\operatorname{supp}u_{l,r}\subset Q_{l,r},$
and
\[
\widehat u_{l,r}(\xi)
=
l^{1/2}r^{(n-1)/2}
\widehat\phi\bigl(l(\xi_1-\rho_s)\bigr)
\widehat\psi(r\xi').
\]

Using the shifted Fourier representation and the change of variables $\zeta_1=l(\xi_1-\rho_s),\,
\zeta'=r\xi',$
we obtain
\begin{align*}
\mathcal E_{s+\ln}(u_{l,r},u_{l,r})
+\frac1{es}\|u_{l,r}\|_2^2
&=
\int_{\R^n}
\tau_s\left(
\rho_se_1+
\left(\frac{\zeta_1}{l},\frac{\zeta'}{r}\right)
\right)
|\widehat\phi(\zeta_1)|^2
|\widehat\psi(\zeta')|^2\,d\zeta.
\end{align*}
By \eqref{eq:anisotropic-symbol},
\[
\tau_s\left(
\rho_se_1+
\left(\frac{\zeta_1}{l},\frac{\zeta'}{r}\right)
\right)
\leq
C_{n,s}
\left(
\frac{\zeta_1^2}{l^2}
+
\frac{|\zeta'|^4}{r^4}
\right).
\]
Since $\widehat\phi$ and $\widehat\psi$ are Schwartz functions,
\[
\int_{\R}\zeta_1^2|\widehat\phi(\zeta_1)|^2\,d\zeta_1<\infty,
\qquad
\int_{\R^{n-1}}
|\zeta'|^4|\widehat\psi(\zeta')|^2\,d\zeta'<\infty.
\]
Hence
\[\mathcal E_{s+\ln}(u_{l,r},u_{l,r})
+\frac1{es}\|u_{l,r}\|_2^2
\leq
C_{n,s}\left(l^{-2}+r^{-4}\right).\]

Although $u_{l,r}$ is complex-valued, write $u_{l,r}=a_{l,r}+ib_{l,r},$
where $a_{l,r}$ and $b_{l,r}$ are real-valued. Since the shifted quadratic
form is real and diagonal with respect to real and imaginary parts,
\[
\begin{aligned}
\mathcal E_{s+\ln}(u_{l,r},u_{l,r})
+\frac1{es}\|u_{l,r}\|_2^2
&=
\mathcal E_{s+\ln}(a_{l,r},a_{l,r})
+\frac1{es}\|a_{l,r}\|_2^2\\
&\quad+
\mathcal E_{s+\ln}(b_{l,r},b_{l,r})
+\frac1{es}\|b_{l,r}\|_2^2,
\end{aligned}
\]
while $\|u_{l,r}\|_2^2
=
\|a_{l,r}\|_2^2+\|b_{l,r}\|_2^2.$
Therefore at least one of $a_{l,r}$ and $b_{l,r}$ is nonzero and has
shifted Rayleigh quotient not exceeding that of $u_{l,r}$. Consequently,
\[
\mu_s(Q_{l,r})
\leq
C_{n,s}^{(2)}\left(l^{-2}+r^{-4}\right).
\]

Finally, for $Q_V$ defined in \eqref{def q}, take $l=V^{\frac{2}{n+1}},\,
r=V^{\frac{1}{n+1}}.$
Then $l^{-2}=r^{-4}=V^{-\frac{4}{n+1}},$
and hence
\[
\mu_s(Q_V)
\leq
C_{n,s}^{(2)}V^{-\frac{4}{n+1}}.
\]
This proves the claim.
\end{proof}

 \subsection{Asymptotics of the First Eigenvalue on Large Balls}
We next derive the large-radius lower bound on balls. The estimate follows
from an uncertainty principle near the minimizing sphere
$|\xi|=\rho_s$: localization of $u$ to $B_R$ gives control of
$\nabla_\xi\widehat u$, while the quadratic lower bound for $\tau_s$ near
$|\xi|=\rho_s$ converts this into a spectral gap of order $R^{-2}$.

\begin{proposition}
\label{prop:ball-lower}
There exists $C_{n,s}^{(3)}>0$ such that, for every $R\geq1$ and every
$u\in\mathcal H_0^{s+\ln}(B_R)$ with $\|u\|_2=1$,
\[\int_{\R^n}
\tau_s(\xi)|\widehat u(\xi)|^2\,d\xi
\geq
C_{n,s}^{(3)}R^{-2},\]
where $\tau_s$ is defined in \eqref{eq:shifted}. Consequently, $\mu_s(B_R)
\geq
C_{n,s}^{(3)}R^{-2}.$
\end{proposition}

\begin{proof}
Let $u\in\mathcal H_0^{s+\ln}(B_R)$ satisfy $\|u\|_2=1$, and set $E
:=
\int_{\R^n}
\tau_s(\xi)|\widehat u(\xi)|^2\,d\xi.$
Since $u=0\,\,a.e.$ in $\mathbb{R}^n\setminus B_R$, we have $xu\in L^2(\R^n)$ and, by Plancherel identity,
\begin{equation}\label{eq:fourier-gradient-ball}
\|\nabla_\xi\widehat u\|_2
=
\|xu\|_2
\leq
R.
\end{equation}

After decreasing $\delta_s$ if necessary, we may assume
$0<\delta_s<\rho_s/2$, where $\delta_s$ is given in Lemma~\ref{lem:symbol-geometry}. Fix $0<\delta_0<\delta_s$ and choose $\chi\in C_c^\infty((0,\infty)),\,0\leq\chi\leq1,$
such that
\[
\chi(r)=1
\quad\text{for }|r-\rho_s|\leq\delta_0,
\qquad
\supp\chi\subset(\rho_s-\delta_s,\rho_s+\delta_s).
\]
We use $\chi(|\xi|)$ as a radial cutoff in Fourier space. By
\eqref{eq:tau-away},
\[
\tau_s(r)\geq c_s\delta_0^2
\qquad
\text{whenever }|r-\rho_s|\geq\delta_0.
\]
Since $0\leq\chi\leq1$ and $1-\chi(r)=0$ for
$|r-\rho_s|\leq\delta_0$, it follows that $(1-\chi(r))^2
\leq C_s\tau_s(r),\,r\ge 0.$

\smallskip

Let $D:=-i\nabla$ and  $T:=\sqrt{-\Delta},$
so that $\widehat{Tu}(\xi)=|\xi|\widehat u(\xi),$
and define
\[
v:=\chi(T)u,
\qquad
\widehat v(\xi)=\chi(|\xi|)\widehat u(\xi).
\]
Therefore
\[
\|u-v\|_2^2
=
\int_{\R^n}
(1-\chi(|\xi|))^2|\widehat u(\xi)|^2\,d\xi
\leq
C_sE.
\]
Fix $\varepsilon_s\in (0,\frac{1}{4C_s})$ sufficiently small. If $E\geq\varepsilon_s$, then,
since $R\geq1$, $E\geq\varepsilon_sR^{-2},$
and there is nothing to prove. We may therefore assume $E<\varepsilon_s$,
and thus
\begin{equation}\label{eq:v-lower-norm}
\|v\|_2
\geq
1-\|u-v\|_2
\geq
\frac12.
\end{equation}

Let $A:=\frac12(x\cdot D+D\cdot x)$
be the generator of the unitary dilation group
\[
(U_t f)(x)
:=
e^{nt/2}f(e^t x),
\qquad t\in\R,
\]
so that $U_t=e^{itA}.$ Since $\widehat v(\xi)=\chi(|\xi|)\widehat u(\xi),$
we have
\[
\nabla_\xi\widehat v
=
\chi(|\xi|)\nabla_\xi\widehat u
+
\chi'(|\xi|)
\frac{\xi}{|\xi|}\widehat u.
\]
Using \eqref{eq:fourier-gradient-ball}, $\|u\|_2=1$, and the fact that
$\chi$ is fixed, we obtain
\[
\|\nabla_\xi\widehat v\|_2
\leq
\|\nabla_\xi\widehat u\|_2
+
\|\chi'\|_\infty\|\widehat u\|_2
\leq
C_s(R+1).
\]
Since $\widehat{Av}(\xi)
=
i\left(\xi\cdot\nabla_\xi+\frac n2\right)\widehat v(\xi),$
the boundedness of $|\xi|$ on $\supp\widehat v$ gives
\begin{equation}\label{eq:Av-bound}
\|Av\|_2
\leq
C_{n,s}\bigl(
\|\nabla_\xi\widehat v\|_2+\|\widehat v\|_2
\bigr)
\leq
C_{n,s}(R+1).
\end{equation}

On the other hand, $T$ acts in Fourier space as multiplication by
$|\xi|$. Hence, using \eqref{eq:v-lower-norm},
\[
\langle Tv,v\rangle
=
\int_{\R^n}|\xi|\,|\widehat v(\xi)|^2\,d\xi
\geq
\frac{\rho_s}{2}\|v\|_2^2
\geq
\frac{\rho_s}{8}.
\]
Since $T$ is homogeneous of degree one, its commutator with the dilation
generator satisfies $i[T,A]=T.$
As the constant $\rho_s$ commutes with $A$, we also have $i[T-\rho_s,A]=T.$
Therefore,
\[
\begin{aligned}
\frac{\rho_s}{8}
&\leq
\langle Tv,v\rangle
=
\langle i[T-\rho_s,A]v,v\rangle\leq
2\|(T-\rho_s)v\|_2\,\|Av\|_2.
\end{aligned}
\]
Combining this with \eqref{eq:Av-bound} yields $\|(T-\rho_s)v\|_2
\geq
C_{n,s}(R+1)^{-1}.$ Finally, since $\supp\widehat v\subset\{|\,|\xi|-\rho_s|<\delta_s\}$,
\eqref{eq:tau-quadratic} gives
\[
\begin{aligned}
E
&\geq
\int_{\R^n}
\tau_s(\xi)|\widehat v(\xi)|^2\,d\xi\geq
c_s
\int_{\R^n}
\bigl(|\xi|-\rho_s\bigr)^2
|\widehat v(\xi)|^2\,d\xi\\
&=
c_s\|(T-\rho_s)v\|_2^2\geq
C_{n,s}^{(3)}(R+1)^{-2}
\geq
C_{n,s}^{(3)}R^{-2},
\end{aligned}
\]
after changing the constant. 
Taking the infimum over all admissible $u$ gives the desired result.
\end{proof}

\begin{remark}
The proof of Proposition~\ref{prop:ball-lower} can be viewed as a radial
uncertainty principle around the minimizing sphere $\Sigma_s=\rho_s\mathbb S^{n-1}.$
The choice
\[
A:=\frac12(x\cdot D+D\cdot x)
\]
is natural because, in Fourier variables, $\widehat A=i\left(\xi\cdot\nabla_\xi+\frac n2\right),$
so that $A$ generates radial dilations. Since $T-\rho_s$ has symbol
$|\xi|-\rho_s$, $\xi\cdot\nabla_\xi\bigl(|\xi|-\rho_s\bigr)=|\xi|,$
and hence $i[T-\rho_s,A]=T.$
Thus, on a fixed annulus around $\Sigma_s$, the commutator is strictly
positive and yields
\[
\|(T-\rho_s)v\|_2\gtrsim R^{-1}.
\]
Since $\tau_s(\xi)\asymp
\bigl||\xi|-\rho_s\bigr|^2$
near $\Sigma_s$, this gives the lower bound of order $R^{-2}$.
\end{remark}

We next derive a matching upper bound on balls by testing the shifted
Rayleigh quotient with functions whose Fourier mass is concentrated near
the minimizing sphere $|\xi|=\rho_s$.

\begin{lemma}
\label{prop:ball-upper}
There exists $C_{n,s}^{(4)}>0$ such that, for every $R\geq1$,
\[
\mu_s(B_R)\leq C_{n,s}^{(4)}R^{-2}.
\]
\end{lemma}

\begin{proof}
Choose a real-valued function $\phi\in C_c^\infty(B_1)$ with
$\|\phi\|_2=1$, and set
\[
u_R(x)
:=
R^{-n/2}e^{i\rho_sx_1}\phi\left(\frac{x}{R}\right).
\]
Then $\supp u_R\subset B_R$, $\|u_R\|_2=1$, and $\widehat u_R(\xi)
=
R^{n/2}
\widehat\phi\bigl(R(\xi-\rho_se_1)\bigr).$
Hence, using the shifted Fourier representation and the change of variables
$\zeta=R(\xi-\rho_se_1)$,
\[
\begin{aligned}
\mathcal E_{s+\ln}(u_R,u_R)+\frac1{es}
&=
\int_{\R^n}
\tau_s\left(\rho_se_1+\frac{\zeta}{R}\right)
|\widehat\phi(\zeta)|^2\,d\zeta\\
&\leq
C_s\int_{\R^n}
\left(
\frac{\zeta_1^2}{R^2}
+
\frac{|\zeta'|^4}{R^4}
\right)
|\widehat\phi(\zeta)|^2\,d\zeta\leq
C_{n,s}^{(4)}R^{-2},
\end{aligned}
\]
where we used \eqref{eq:anisotropic-symbol} and $R\geq1$ in the last
inequality. By passing to the real or imaginary part as in
Proposition~\ref{prop:box-upper}, this yields $\mu_s(B_R)\leq C_{n,s}^{(4)}R^{-2}.$
\end{proof}

\subsection{Failure of the Faber--Krahn Inequality in the Large-Volume Regime}

We now combine the large-volume lower bound for balls with the matching upper
bound on the elongated boxes $Q_V$ to show that the ball ceases to minimize the
first Dirichlet eigenvalue when the volume is sufficiently large.

\begin{proof}[\textbf{Proof of Theorem~\ref{teo dom-2}.}]
By Proposition~\ref{prop:volume-lower},
there exist $c_{n,s}>0$ and $V_0>0$ such that
\[
\mu_s(\Omega)
\geq
c_{n,s}|\Omega|^{-\frac{4}{n+1}}
\]
for every bounded open set $\Omega\subset\R^n$ with $|\Omega|\geq V_0$.
Since
\[
\mu_s(\Omega)
=
\lambda_1^{s+\ln}(\Omega)+\frac1{es},
\]
this proves $(i)$. We now prove $(ii)$. By Lemma~\ref{prop:box-upper}, there exist
$C_{n,s}>0$ and $V_2>0$ such that
\begin{equation}\label{eq:box-large-upper}
\mu_s(Q_V)
\leq
C_{n,s}V^{-\frac{4}{n+1}}
\qquad
\text{for }V\geq V_2.
\end{equation}
Since
\[
|Q_V|
=
2^n
V^{\frac{2}{n+1}}
\left(V^{\frac{1}{n+1}}\right)^{n-1}
=
2^nV,
\]
the radius $R_V$ of $B_V=B_{R_V}$ satisfies
\[
\omega_nR_V^n=2^nV,
\qquad
R_V
=
\left(\frac{2^nV}{\omega_n}\right)^{1/n}.
\]
Hence Proposition~\ref{prop:ball-lower} yields
\begin{equation}\label{eq:ball-large-lower}
\mu_s(B_V)
\geq
c_{n,s}R_V^{-2}
=
c_{n,s}'
V^{-\frac2n}
\end{equation}
for all sufficiently large $V$. Since $n\geq2$, $\frac{4}{n+1}>\frac2n,$
and therefore
\[
V^{-\frac{4}{n+1}}
=
o\left(V^{-\frac2n}\right)
\qquad
\text{as }V\to\infty.
\]
Combining \eqref{eq:box-large-upper} and
\eqref{eq:ball-large-lower}, we may choose
$V_1=V_1(n,s)\geq\max\{V_0,V_2\}$ such that
\[
\mu_s(Q_V)
<
\mu_s(B_V)
\qquad
\text{for every }V\geq V_1.
\]
Using again $\mu_s(\Omega)
=
\lambda_1^{s+\ln}(\Omega)+\frac1{es},$
we obtain
\[
\lambda_1^{s+\ln}(Q_V)
<
\lambda_1^{s+\ln}(B_V),
\]
which completes the proof.
\end{proof}

\section{Faber--Krahn Inequality in the Small-Volume Regime}

In this section, we prove the Faber--Krahn inequality for the
fractional-logarithmic Laplacian on domains of sufficiently small volume
and characterize the equality case. 

We first record the scaling relation
for the first Dirichlet eigenvalue on balls. Under the dilation
$B_R=RB_1$, the logarithmic factor produces an additional multiple of the
fractional energy, which is encoded by the auxiliary function $F$ below.
We then combine this scaling formula with rearrangement estimates to prove
the small-volume inequality and its rigidity statement.

\begin{lemma}\label{lem:scaling}
Let $n\geq1$, $0<s<1$, and let $B_R\subset\R^n$ be the ball of radius
$R>0$. Define
\[F(L):=
\inf_{\substack{
w\in\mathcal H_0^{s+\ln}(B_1)\\
\|w\|_2=1
}}
\left\{
\int_{\R^n}|\xi|^{2s}\ln|\xi|^2\,|\widehat w(\xi)|^2\,d\xi
+
L\int_{\R^n}|\xi|^{2s}|\widehat w(\xi)|^2\,d\xi
\right\},
\qquad L\in\R.\]
Then, for every $R>0$,
\[\lambda_1^{s+\ln}(B_R)
=
R^{-2s}
F\left(2\ln\frac1R\right).\]
\end{lemma}

\begin{proof}
For $u\in\mathcal H_0^{s+\ln}(B_R)$, define $w(y):=R^{n/2}u(Ry).$
The map $u\mapsto w$ is a bijection from
$\mathcal H_0^{s+\ln}(B_R)$ onto
$\mathcal H_0^{s+\ln}(B_1)$ and preserves the $L^2$-norm $\|w\|_2=\|u\|_2.$
Moreover, $\widehat u(\xi)
=
R^{n/2}\widehat w(R\xi).$
Hence, setting $\eta=R\xi$,
\begin{align*}
\int_{\R^n}
|\xi|^{2s}\ln|\xi|^2\,|\widehat u(\xi)|^2\,d\xi=
R^{-2s}
\int_{\R^n}
|\eta|^{2s}
\left(
\ln|\eta|^2+2\ln\frac1R
\right)
|\widehat w(\eta)|^2\,d\eta.
\end{align*}
Therefore,
\[
\frac{\cE_{s+\ln}(u,u)}{\|u\|_2^2}
=
R^{-2s}
\frac{
\displaystyle
\int_{\R^n}|\eta|^{2s}\ln|\eta|^2|\widehat w(\eta)|^2\,d\eta
+
2\ln\frac1R
\int_{\R^n}|\eta|^{2s}|\widehat w(\eta)|^2\,d\eta
}{
\|w\|_2^2
}.
\]
Taking the infimum over all nonzero
$u\in\mathcal H_0^{s+\ln}(B_R)$, equivalently over all nonzero
$w\in\mathcal H_0^{s+\ln}(B_1)$, yields
\[
\lambda_1^{s+\ln}(B_R)
=
R^{-2s}
F\left(2\ln\frac1R\right),
\]
which completes the proof.
\end{proof}

We are now ready to prove the small-volume Faber--Krahn inequality and characterize the equality case.

\begin{proof}[\textbf{Proof of Theorem~\ref{teo dom-1}.}]
Recall from Section~\ref{preli} that $K_{s+\ln}=K_{s+\ln}^+-K_{s+\ln}^-$ and set
\[
M_s:=\|K_{s+\ln}^-\|_{L^\infty(\R^n)}<\infty,
\qquad
\kappa_s:=\|K_{s+\ln}^-\|_{L^1(\R^n)}<\infty.
\]
Let $V=|\Omega|$ and let $B_{R_V}$ be the ball satisfying $|B_{R_V}|=V$. We shall choose $R_V>0$ sufficiently small, depending only on $n$ and $s$. Take $u\in\mathcal H_0^{s+\ln}(\Omega)$ real-valued with $\|u\|_2=1$.

\smallskip

\textbf{Step 1.} We first record a rearrangement estimate for nonnegative functions. Let $f\in\mathcal H^{s+\ln}(\R^n)$ be nonnegative and  nontrivial, assume $v:=|\{f>0\}|\leq V$. If $2R_V\leq r_{n,s}$, then the ball $B_{R_v}$ of volume $v$ has diameter at most $r_{n,s}$. Recall that
\[
\mathcal I_\pm(f)
=
\frac12
\iint_{\R^n\times\R^n}
(f(x)-f(y))^2K_{s+\ln}^\pm(|x-y|)\,dx\,dy.
\]
Since $K_{s+\ln}^+$ is singular at the origin, for $M>0$ define
\[
K_{s+\ln,M}^+(r):=\min\{K_{s+\ln}^+(r),M\}.
\]
Then $K_{s+\ln,M}^+$ is nonnegative, radial, nonincreasing, and belongs to $L^1(\R^n)$. Set
\[
\mathcal I_{+,M}(f)
:=
\frac12
\iint_{\R^n\times\R^n}
(f(x)-f(y))^2K_{s+\ln,M}^+(|x-y|)\,dx\,dy.
\]
For $f\geq0$, expanding the square gives
\[
\mathcal I_{+,M}(f)
=
\|f\|_2^2\int_{\R^n}K_{s+\ln,M}^+(|z|)\,dz
-
\iint_{\R^n\times\R^n}
f(x)f(y)K_{s+\ln,M}^+(|x-y|)\,dx\,dy,
\]
where the interaction term is finite by Young's inequality. Since
$\|f^\ast\|_2=\|f\|_2$, the Riesz rearrangement inequality
\cite[Theorem~3.9]{LiebLoss2001} yields
\[
\iint
f(x)f(y)K_{s+\ln,M}^+(|x-y|)\,dx\,dy
\leq
\iint
f^\ast(x)f^\ast(y)K_{s+\ln,M}^+(|x-y|)\,dx\,dy,
\]
because $K_{s+\ln,M}^+$ is symmetric decreasing. Hence $\mathcal I_{+,M}(f^\ast)\leq\mathcal I_{+,M}(f).$
Finally, since
\[
K_{s+\ln,M}^+(r)\uparrow K_{s+\ln}^+(r)
\qquad\text{as }M\to\infty,
\]
the monotone convergence theorem gives $\mathcal I_+(f^\ast)\leq\mathcal I_+(f).$ Note that $$\mathcal I_-(f)
\leq
\|f\|_2^2\int_{\R^n}K_{s+\ln}^-(|z|)\,dz.$$
On the other hand, $\operatorname{supp}f^\ast\subset B_{R_v}$ and $\operatorname{diam}(B_{R_v})\leq r_{n,s}$, hence $K_{s+\ln}^-(|x-y|)=0$ on $\operatorname{supp}f^\ast\times\operatorname{supp}f^\ast$, and therefore $\mathcal I_-(f^\ast)=\|f\|_2^2\int_{\R^n}K_{s+\ln}^-(|z|)\,dz.$
Consequently,
\begin{equation}
\label{eq:nonnegative-rearrangement-estimate}
\mathcal E_{s+\ln}(f,f)
\geq
\mathcal E_{s+\ln}(f^\ast,f^\ast)
\geq
\lambda_1^{s+\ln}(B_{R_v})\|f\|_2^2.
\end{equation}

\textbf{Step 2.} We next quantify the dependence of the ball eigenvalue on its volume. By Lemma \ref{lem:scaling}, we obtain
\[
\lambda_1^{s+\ln}(B_R)
=
R^{-2s}F(L),
\qquad
L:=2\ln\frac1R.
\]
Since $r^{2s}\ln r^2\geq-1/(es)$ for every $r\ge 0,$
we have $F(L)\geq L\Lambda_s-\frac1{es}.$
Hence, if $L\geq2/(es\Lambda_s)$, then $F(L)\geq\frac{\Lambda_s}{2}L.$
Moreover, for every $\delta\geq0$, $F(L+\delta)\geq F(L)+\delta\Lambda_s.$

\smallskip

Let $0<a\leq1$. The ball of volume $aV$ has radius ${R_V}a^{1/n}$. Set
\[
p:=\frac{2s}{n},
\qquad
\delta:=\frac{2}{n}\ln\frac1a\geq0,\qquad \Lambda_{n,s}:=\frac{p\Lambda_s}{2},
\]
the function $x\mapsto x^{-p}$ is convex on $(0,\infty)$, we obtain $a^{-p}\geq 1+p(1-a),$ and hence
\begin{equation}\label{eq:ball-eigenvalue-volume-gap}
\begin{aligned}
\lambda_1^{s+\ln}(B_{R_{aV}})-\lambda_1^{s+\ln}(B_{R_V})
&=R_V^{-2s}\bigl[a^{-p}F(L+\delta)-F(L)\bigr]\\&\geq R_V^{-2s}\bigl[(a^{-p}-1)F(L)+a^{-p}\delta\Lambda_s\bigr]\\
&\geq\frac{\Lambda_s}{2}R_V^{-2s}L(a^{-p}-1)
\geq \Lambda_{n,s}R_V^{-2s}L(1-a).
\end{aligned}
\end{equation}

Now write $u=f-g,\,f:=u_+,\,g:=u_-.$
The truncation inequality $|u_\pm(x)-u_\pm(y)|\leq|u(x)-u(y)|$ shows that $f,g\in\mathcal H^{s+\ln}(\R^n)$. Set
\[
\alpha:=\|f\|_2^2,\qquad
\beta:=\|g\|_2^2,\qquad
|\{f>0\}|=aV,\qquad
|\{g>0\}|=bV.
\]
Then $\alpha+\beta=1$ and $a+b\leq1.$
If $\alpha\beta=0$, then $u$ has one sign and \eqref{eq:nonnegative-rearrangement-estimate}, together with $a\leq1$ and \eqref{eq:ball-eigenvalue-volume-gap}, immediately yields $\mathcal E_{s+\ln}(u,u)\geq\lambda_1^{s+\ln}(B_{R_V}).$

\smallskip

Assume henceforth that $\alpha,\beta>0$. By \eqref{eq:nonnegative-rearrangement-estimate} and \eqref{eq:ball-eigenvalue-volume-gap},
\begin{equation}
\label{eq:positive-negative-energy-lower-bound}
\begin{aligned}
\mathcal E_{s+\ln}(f,f)+\mathcal E_{s+\ln}(g,g)
\geq\lambda_1^{s+\ln}(B_{R_V})
+\Lambda_{n,s}R_V^{-2s}L\bigl[\alpha(1-a)+\beta(1-b)\bigr].
\end{aligned}
\end{equation}
Since $fg=0$ a.e.,
\[
\mathcal E_{s+\ln}(f,g)
=
-\iint_{\R^n\times\R^n}
f(x)g(y)K_{s+\ln}(|x-y|)\,dx\,dy,
\]
and hence
\[
\mathcal E_{s+\ln}(f,g)
\leq
\iint f(x)g(y)K_{s+\ln}^-(|x-y|)\,dx\,dy
\leq
M_s\|f\|_1\|g\|_1.
\]
By Cauchy--Schwarz, $\|f\|_1\leq(aV\alpha)^{1/2},\,\|g\|_1\leq(bV\beta)^{1/2},$
so
\begin{equation}
\label{eq:cross-interaction-upper-bound}
\mathcal E_{s+\ln}(f,g)
\leq
M_sV\sqrt{ab\alpha\beta}.
\end{equation}
Since $a+b\leq1$,
\begin{equation}
\label{eq:mass-volume-defect-lower-bound}
\alpha(1-a)+\beta(1-b)
\geq
\alpha b+\beta a
\geq
2\sqrt{ab\alpha\beta}.
\end{equation}
Combining \eqref{eq:positive-negative-energy-lower-bound}--\eqref{eq:mass-volume-defect-lower-bound}, we obtain
\begin{equation}
\label{eq:signed-function-energy-lower-bound}
\begin{aligned}
\mathcal E_{s+\ln}(u,u)
&=
\mathcal E_{s+\ln}(f,f)+\mathcal E_{s+\ln}(g,g)
-2\mathcal E_{s+\ln}(f,g)\\
&\geq
\lambda_1^{s+\ln}(B_{R_V})
+
2\left(
\Lambda_{n,s}R_V^{-2s}L-M_sV
\right)\sqrt{ab\alpha\beta}.
\end{aligned}
\end{equation}
Since $V=|B_1|R_V^n$ and $L=2\ln(1/R_V)$, $R_V^{-2s}L\rightarrow\infty,\,
V\rightarrow0$ as $R_V\to0.$
We may therefore choose $R_0=R_0(n,s)>0$ such that
\[ 2R_0< r_{n,s},\quad
2\ln\frac1{R_0}\geq\frac{2}{es\Lambda_s},
\quad
\Lambda_{n,s} R_V^{-2s}2\ln\frac{1}{R_V}
>
M_s|B_1|R_V^n
\quad\text{for }0<R_V\leq R_0.\]
Setting $V_0:=|B_1|R_0^n$, \eqref{eq:signed-function-energy-lower-bound} gives $\mathcal E_{s+\ln}(u,u)
\geq
\lambda_1^{s+\ln}(B_{R_V})$
whenever $V\leq V_0$. Taking the infimum over all $u\in\mathcal H_0^{s+\ln}(\Omega)$ with $\|u\|_2=1$ yields $\lambda_1^{s+\ln}(\Omega)
\geq
\lambda_1^{s+\ln}(B_{R_V}).$

\smallskip

\textbf{Step 3.} It remains to prove the rigidity statement. Suppose that $\lambda_1^{s+\ln}(\Omega)
=
\lambda_1^{s+\ln}(B_{R_V})$ with $|\Omega|=|B_{R_V}|=V$
and let $u\in\mathcal H_0^{s+\ln}(\Omega)$ be a real-valued  first eigenfunction with $\|u\|_2=1$. With the notation $u=f-g,\, f=u_+,\, g=u_-,$
the preceding estimate gives $0
\geq
2\bigl(\Lambda_{n,s}R_V^{-2s}L-M_sV\bigr)
\sqrt{ab\alpha\beta}.$
Hence $ab\alpha\beta=0$, and therefore $u$ has a constant sign. Replacing $u$ by $-u$ if necessary, we may assume $u\geq0$. Writing $|\{u>0\}|=aV$, \eqref{eq:nonnegative-rearrangement-estimate} and \eqref{eq:ball-eigenvalue-volume-gap} yield
\[
\lambda_1^{s+\ln}(B_{R_V})
=
\mathcal E_{s+\ln}(u,u)
\geq
\lambda_1^{s+\ln}(B_{R_{aV}})
\geq
\lambda_1^{s+\ln}(B_{R_V})
+
\Lambda_{n,s}R_V^{-2s}L(1-a),
\]
and hence $a=1$. Thus $u>0$ a.e. in $\Omega.$

\smallskip

Let $u^\ast$ be the symmetric decreasing rearrangement of $u$. Since
$|\{u>0\}|=V$, equimeasurability and the radial monotonicity of $u^\ast$
imply that $u^\ast=0$ a.e. in $B_{R_V}^c.$ Moreover, $\|u^\ast\|_2=\|u\|_2=1.$
Hence $u^\ast$ is an admissible test function for
$\lambda_1^{s+\ln}(B_{R_V})$, and the rearrangement argument above yields
\[
\mathcal E_{s+\ln}(u,u)
\geq
\mathcal E_{s+\ln}(u^\ast,u^\ast)
\geq
\lambda_1^{s+\ln}(B_{R_V}).
\]
Thus, $\mathcal E_{s+\ln}(u,u)
=
\mathcal E_{s+\ln}(u^\ast,u^\ast).$
Since
\[
\mathcal E_{s+\ln}(v,v)
=
\mathcal I_+(v)-\mathcal I_-(v),
\qquad
\mathcal I_+(u)\geq\mathcal I_+(u^\ast),
\qquad
\mathcal I_-(u)\leq\mathcal I_-(u^\ast),
\]
we obtain that $\mathcal I_+(u)=\mathcal I_+(u^\ast)$ and $\mathcal I_-(u)=\mathcal I_-(u^\ast).$
Since $2R_V<r_{n,s}$, 
\[
\mathcal I_-(u^\ast)
=
\|u^\ast\|_2^2\int_{\R^n}K_{s+\ln}^-(|z|)\,dz
=
\int_{\R^n}K_{s+\ln}^-(|z|)\,dz.
\]
Consequently,
\begin{equation}
\label{eq:rigidity-negative-interaction-zero}
\iint_{\R^n\times\R^n}
u(x)u(y)K_{s+\ln}^-(|x-y|)\,dx\,dy
=
0.
\end{equation}

We claim that $\operatorname{diam}(\Omega)\leq r_{n,s}$. Otherwise there exist $x_0,y_0\in\Omega$ with $|x_0-y_0|>r_{n,s}$. Since $\Omega$ is open, one can choose $\varepsilon>0$ such that
\[
B_\varepsilon(x_0),\,B_\varepsilon(y_0)\subset\Omega,
\qquad
|x-y|>r_{n,s}
\quad
\text{for all }(x,y)\in
B_\varepsilon(x_0)\times B_\varepsilon(y_0).
\]
Since $u>0$ a.e. in $\Omega$ and $K_{s+\ln}^-(r)>0$ for $r>r_{n,s}$, we obtain
\[
\begin{aligned}
&\iint_{\R^n\times\R^n}
u(x)u(y)K_{s+\ln}^-(|x-y|)\,dx\,dy\geq
\iint_{B_\varepsilon(x_0)\times B_\varepsilon(y_0)}
u(x)u(y)K_{s+\ln}^-(|x-y|)\,dx\,dy
>0,
\end{aligned}
\]
contradicting \eqref{eq:rigidity-negative-interaction-zero}. Hence $\operatorname{diam}(\Omega)\leq r_{n,s}.$

\smallskip

It remains to identify the equality case in the positive-kernel rearrangement. By the Riesz rearrangement inequality, $\mathcal I_{+,M}(u^\ast)
\leq
\mathcal I_{+,M}(u)$ for every $M>0.$
Moreover, if $M_2>M_1$, then
\[
K_{s+\ln,M_2}^+-K_{s+\ln,M_1}^+
\]
is nonnegative, radial, and nonincreasing. Applying the Riesz rearrangement inequality to this kernel shows that
\[
0\leq
\mathcal I_{+,M_1}(u)-\mathcal I_{+,M_1}(u^\ast)
\leq
\mathcal I_{+,M_2}(u)-\mathcal I_{+,M_2}(u^\ast).
\]
Since $\mathcal I_{+,M}(u)\uparrow\mathcal I_+(u),\,\mathcal I_{+,M}(u^\ast)\uparrow\mathcal I_+(u^\ast),$
and $\mathcal I_+(u)=\mathcal I_+(u^\ast)$, it follows that
\[
\mathcal I_{+,M}(u)
=
\mathcal I_{+,M}(u^\ast)
\qquad\text{for every }M>0.
\]
Using the expansion of $\mathcal I_{+,M}$ and $\|u\|_2=\|u^\ast\|_2$, we therefore obtain
\[
\iint_{\R^n\times\R^n}
u(x)u(y)K_{s+\ln,M}^+(|x-y|)\,dx\,dy
=
\iint_{\R^n\times\R^n}
u^\ast(x)u^\ast(y)K_{s+\ln,M}^+(|x-y|)\,dx\,dy
\]
for every $M>0$. Now set
\[
H_0(r)
:=
\int_0^\infty e^{-M}K_{s+\ln,M}^+(r)\,dM.
\]
Since $H_0(r)
=
1-e^{-K_{s+\ln}^+(r)}$ for $0<r<r_{n,s},$
and $K_{s+\ln}^+$ is strictly decreasing on $(0,r_{n,s})$, the function
$H_0$ is bounded, nonnegative, supported in $[0,r_{n,s}]$, and strictly
decreasing on $(0,r_{n,s})$. Integrating the preceding identity with
respect to $e^{-M}\,dM$ gives
\[
\iint_{\R^n\times\R^n}
u(x)u(y)H_0(|x-y|)\,dx\,dy
=
\iint_{\R^n\times\R^n}
u^\ast(x)u^\ast(y)H_0(|x-y|)\,dx\,dy.
\]

Fix $c>0$ and define
\[
H(r)
:=
\begin{cases}
H_0(r)+c,&0\leq r\leq r_{n,s},\\
ce^{-(r-r_{n,s})},&r>r_{n,s}.
\end{cases}
\]
Then $H\in L^1(\R^n)$ is nonnegative, radial, and strictly decreasing.
Since $\operatorname{diam}(\Omega)\leq r_{n,s}$ and
$2R_V<r_{n,s}$, all distances between points in the essential supports
of $u$ and $u^\ast$ are at most $r_{n,s}$. Moreover,
$\int_{\R^n}u\,dx=\int_{\R^n}u^\ast\,dx$ by equimeasurability. Hence
\[
\iint_{\R^n\times\R^n}
u(x)u(y)H(|x-y|)\,dx\,dy
=
\iint_{\R^n\times\R^n}
u^\ast(x)u^\ast(y)H(|x-y|)\,dx\,dy.
\]
The strict equality case in the Riesz rearrangement inequality
\cite[Theorem~3.9]{LiebLoss2001} therefore yields some $x_0\in\R^n$ such that
\[
u(x)=u^\ast(x-x_0)
\qquad\text{for a.e. }x\in\R^n.
\]
Since $u>0$ a.e. in $\Omega$ and $u=0$ a.e. in $\Omega^c$, while
$\{u^\ast>0\}=B_{R_V}$ up to a null set, we conclude that $|\Omega\triangle(B_{R_V}+x_0)|=0.$
Thus equality can occur only when $\Omega$ coincides, up to a null set,
with a ball. If $\Omega$ has continuous boundary, then $\Omega=B_{R_V}+x_0,$
which completes the proof.
\end{proof}


An additional consequence of the small-volume result is that it recovers the classical Faber--Krahn inequality for the fractional Laplacian.

\begin{corollary}
\label{cor:fractional-fk-from-fraclog}
Let $n\geq1$ and $s\in(0,1)$. Then, for every bounded open set
$\Omega\subset\R^n$,
\[
\lambda_{1,s}(\Omega)
\geq
\lambda_{1,s}(B),
\]
where $B$ is any ball satisfying $|B|=|\Omega|$.
\end{corollary}

\begin{proof}
For $t>0$ sufficiently small, $|t\Omega|\leq V_0$, and hence
Theorem~\ref{teo dom-1} gives $$\lambda_1^{s+\ln}(t\Omega)
\geq
\lambda_1^{s+\ln}(tB).$$
For a bounded open set $D$, set
\[
F_L(D)
:=
\inf_{\substack{u\in\mathcal H_0^{s+\ln}(D)\\ \|u\|_2=1}}
\left\{
\int_{\R^n}|\xi|^{2s}\ln|\xi|^2|\widehat u(\xi)|^2\,d\xi
+
L\int_{\R^n}|\xi|^{2s}|\widehat u(\xi)|^2\,d\xi
\right\}.
\]
The $L^2$-preserving dilation gives $\lambda_1^{s+\ln}(tD)
=
t^{-2s}F_L(D),\,L=2\ln\frac1t.$
Thus
\[
F_L(\Omega)\geq F_L(B)
\qquad\text{for all sufficiently large }L.
\]
Moreover, $\displaystyle \lim_{L\to\infty}\frac{F_L(D)}{L}
=
\lambda_{1,s}(D).$
Indeed, since $r^{2s}\ln r^2\geq-\frac1{es},\,r\ge 0,$
one has
\[
\frac{F_L(D)}L
\geq
\lambda_{1,s}(D)-\frac1{esL}.
\]
Conversely, for every $\varepsilon>0$, choose
$\varphi\in C_c^\infty(D)$ with $\|\varphi\|_2=1$ and
\[
\int_{\R^n}|\xi|^{2s}|\widehat\varphi(\xi)|^2\,d\xi
\leq
\lambda_{1,s}(D)+\varepsilon.
\]
Then
\[
\frac{F_L(D)}L
\leq
\lambda_{1,s}(D)+\varepsilon
+
\frac1L
\int_{\R^n}|\xi|^{2s}\ln|\xi|^2
|\widehat\varphi(\xi)|^2\,d\xi.
\]
Letting first $L\to\infty$ and then $\varepsilon\to0$ proves the claim.
Therefore, dividing $F_L(\Omega)\geq F_L(B)$ by $L$ and passing to the
limit $L\to\infty$, we obtain $\lambda_{1,s}(\Omega)\geq\lambda_{1,s}(B).$
\end{proof}


\section{Appendix: rigidity for the Logarithmic Faber--Krahn inequality} 
The purpose of this section is to settle the equality case in the
Faber--Krahn inequality for the logarithmic Laplacian. We use the notation of \cite{ChenWeth2019}. For a bounded Lipschitz domain
$\Omega\subset\R^N$, let $\lambda_1^{\ln}(\Omega)$ denote the first
Dirichlet eigenvalue of $(-\Delta)^{\ln}$ with zero exterior condition
$u=0$ a.e.\ in $\R^N\setminus\Omega$. Its quadratic form is
\begin{equation}\label{eq:log-form}
\begin{aligned}
\cE_L(u,u)
={}&
\frac{c_N}{2}
\iint_{|x-y|<1}
\frac{(u(x)-u(y))^2}{|x-y|^N}\,dx\,dy \\
&-
c_N
\iint_{|x-y|\geq1}
\frac{u(x)u(y)}{|x-y|^N}\,dx\,dy
+\rho_N\|u\|_2^2,
\end{aligned}
\end{equation}
where $c_N>0$ and $\rho_N\in\R$ depend only on $N$.

Chen and Weth
\cite[Corollary 3.6]{ChenWeth2019} proved that, for every bounded Lipschitz domain
$\Omega\subset\R^N$ and every ball $B$ satisfying $|B|=|\Omega|$,
\[
\lambda_1^{\ln}(\Omega)\geq \lambda_1^{\ln}(B).
\]
Their proof is obtained by passing to the limit from the fractional
Faber--Krahn inequality. As pointed out in
\cite[after Corollary~3.6]{ChenWeth2019}, this argument does not determine
the equality case, and the rigidity statement
\[
\lambda_1^{\ln}(\Omega)=\lambda_1^{\ln}(B)
\quad\Longrightarrow\quad
\Omega \text{ is a ball, up to translation}
\]
was left open. We now establish the rigidity statement.

\begin{lemma}\label{lem:scaling1}
Let $\Omega\subset\R^N$ be a bounded Lipschitz domain. Then, for every
$r>0$,
\[
\lambda_1^{\ln}(r\Omega)
=
\lambda_1^{\ln}(\Omega)-2\log r.
\]
In particular, for any bounded Lipschitz domains $\Omega_1,\Omega_2$,
\[
\lambda_1^{\ln}(r\Omega_1)-\lambda_1^{\ln}(r\Omega_2)
=
\lambda_1^{\ln}(\Omega_1)-\lambda_1^{\ln}(\Omega_2).
\]
\end{lemma}

\begin{proof}
For $u\in C_c^\infty(\Omega)$, define $u_r(x):=r^{-N/2}u(x/r).$
Then $u_r\in C_c^\infty(r\Omega)$, $\|u_r\|_2=\|u\|_2,\,
\widehat{u_r}(\xi)=r^{N/2}\widehat u(r\xi).$
Using the Fourier representation of the logarithmic Dirichlet form,
\[
\cE_L(u,u)
=
\int_{\R^N}2\log|\xi|\,|\widehat u(\xi)|^2\,d\xi,
\]
we obtain, after the change of variables $\eta=r\xi$,
\begin{align*}
\cE_L(u_r,u_r)
&=
\int_{\R^N}
2\log|\xi|\,r^N|\widehat u(r\xi)|^2\,d\xi=
\int_{\R^N}
2\log\frac{|\eta|}{r}\,
|\widehat u(\eta)|^2\,d\eta=
\cE_L(u,u)-2\log r\,\|u\|_2^2.
\end{align*}
By density, the dilation $u\mapsto u_r$ extends to a unitary
bijection between the Dirichlet form spaces associated with $\Omega$
and $r\Omega$, and the same identity holds throughout the form
domain. Therefore,
\begin{align*}
\lambda_1^{\ln}(r\Omega)=
\lambda_1^{\ln}(\Omega)-2\log r.
\end{align*}
The second assertion follows immediately.
\end{proof}

 Set $K(z):=|z|^{-N}\mathbf 1_{\{|z|<1\}}$ and 
\[
\mathcal D_K(u):=
\iint_{\R^N\times\R^N}
(u(x)-u(y))^2K(x-y)\,dx\,dy.
\]
For a nonnegative measurable function $u$, we denote by $u^*$ its
symmetric decreasing rearrangement.

For a measurable function $u$, we write
\[
\operatorname{ess\,supp}u
:=
\R^N\setminus
\bigcup\left\{
U\subset\R^N:\ U\text{ is open and }u=0\text{ a.e. in }U
\right\}.
\]

We next isolate the strict rearrangement argument. The proof consists of two steps. We first regularize the singular kernel
$K$ by truncation and apply the Riesz rearrangement inequality to obtain
the monotonicity of $\mathcal D_K$ under symmetric decreasing
rearrangement. For the equality case, we decompose $K$ into a sum of two
radial nonincreasing kernels, one of which is strictly decreasing on the
relevant distance range. The assumption
$\operatorname{diam}(\operatorname{ess\,supp}u)<1$ then allows us to
replace this kernel by a globally strictly decreasing integrable one,
so that the equality characterization in the Riesz rearrangement
inequality applies.

\begin{proposition}\label{prop:strict}
Let $u\geq0$ be nonzero and compactly supported, and assume that
\[
d:=\operatorname{diam}(\operatorname{ess\,supp}u)<1,
\qquad
\mathcal D_K(u)<\infty.
\]
Then $\mathcal D_K(u^*)\leq \mathcal D_K(u).$
If equality holds, then there exists $a\in\R^N$ such that
\[
u(x)=u^*(x-a)
\qquad\text{for a.e. }x\in\R^N.
\]
\end{proposition}

\begin{proof}
We first note that $u\in L^2(\R^N)$. Indeed, if
$E:=\operatorname{ess\,supp}u$, then for every $x\in E$ and almost every
$y$ satisfying $d<|x-y|<1,$
we have $u(y)=0$. Hence
\[
\begin{aligned}
\mathcal D_K(u)
&\geq
\int_E u(x)^2
\int_{\{d<|x-y|<1\}}
\frac{dy}{|x-y|^N}\,dx=
|\mathbb S^{N-1}|
\log\frac1d\,
\|u\|_2^2.
\end{aligned}
\]

For $m>0$, define $K_m:=\min\{K,m\}.$
Then $K_m\in L^1(\R^N)$ is nonnegative, radial, and nonincreasing. Since
\[
\mathcal D_{K_m}(u)
=
2\|K_m\|_1\|u\|_2^2
-
2\iint_{\R^N\times\R^N}
u(x)u(y)K_m(x-y)\,dx\,dy,
\]
the Riesz rearrangement inequality gives $\mathcal D_{K_m}(u^*)
\leq
\mathcal D_{K_m}(u).$
Since $K_m\uparrow K$ almost everywhere, monotone convergence yields
\begin{equation}\label{eq:PSK}
\mathcal D_K(u^*)
\leq
\mathcal D_K(u).
\end{equation}

Assume now that equality holds in \eqref{eq:PSK}. Fix
$0<\varepsilon<1$ and define
\[
J(z):=\varepsilon(1-|z|)_+,
\qquad
R(z):=K(z)-J(z).
\]
For $0<r<1$, $R(r)=r^{-N}-\varepsilon(1-r),$
and
\[
R'(r)
=
-Nr^{-N-1}+\varepsilon
\leq
-N+\varepsilon<0.
\]
Moreover, $R(1^-)=1,\,
R(r)=0$ for $r\geq1.$
Thus $R$ is nonnegative, radial, and nonincreasing. Applying the preceding
truncation argument to $R$, and the usual Riesz rearrangement inequality
to $J$, we obtain
\[
\mathcal D_R(u^*)\leq\mathcal D_R(u),
\qquad
\mathcal D_J(u^*)\leq\mathcal D_J(u).
\]
Since $\mathcal D_K=\mathcal D_J+\mathcal D_R,$
equality in \eqref{eq:PSK} implies $\mathcal D_J(u^*)
=
\mathcal D_J(u).$
Because $J\in L^1(\R^N)$ and $\|u^*\|_2=\|u\|_2,$ thus
\begin{equation}\label{eq:IJ-eq}
\iint_{\R^N\times\R^N}
u(x)u(y)J(x-y)\,dx\,dy
=
\iint_{\R^N\times\R^N}
u^*(x)u^*(y)J(x-y)\,dx\,dy.
\end{equation}

Let $E:=\{u>0\}.$
Since $\operatorname{diam}(E)\leq d$, the isodiametric inequality gives $|E|
\leq
\left|B_{d/2}\right|,$ see \cite[Section~6]{Fusco2015}.
Because $\{u^*>0\}$ is a centered ball with measure $|E|$, it follows that $\operatorname{diam}(\operatorname{supp}u^*)
\leq d.$

\smallskip

We now replace $J$ by a strictly radial decreasing integrable kernel.
Define $\widetilde J(z):=\widetilde j(|z|),$
where
\[
\widetilde j(r):=
\begin{cases}
\varepsilon(1-r),
&0\leq r\leq d,\\[1mm]
\varepsilon(1-d)e^{-(r-d)},
&r>d.
\end{cases}
\]
Then $\widetilde J$ is positive, radial, strictly decreasing, and belongs
to $L^1(\R^N)$. Moreover, whenever
\[
u(x)u(y)\neq0
\quad\text{or}\quad
u^*(x)u^*(y)\neq0,
\]
we have $|x-y|\leq d$. Hence $J(x-y)=\widetilde J(x-y)$
on all pairs contributing to the two integrals in
\eqref{eq:IJ-eq}. Therefore,
\[
\iint
u(x)u(y)\widetilde J(x-y)\,dx\,dy
=
\iint
u^*(x)u^*(y)\widetilde J(x-y)\,dx\,dy.
\]
Since $\widetilde J$ is strictly radially decreasing, the strict equality
case in the Riesz rearrangement inequality
\cite[Theorem 3.9]{LiebLoss2001} implies that there exists $a\in\R^N$ such that
\[
u(x)=u^*(x-a)
\qquad\text{for a.e. }x\in\R^N,
\]
which completes the proof.
\end{proof}

\begin{proof}[\textbf{Proof of Theorem~\ref{thm:rigidity}.}]
By Lemma~\ref{lem:scaling1}, it is enough to prove the assertion under the
additional assumption ${\rm diam}(\Omega)<1.$

Let $u_1\ge0$ be the $L^2$-normalized first Dirichlet eigenfunction of
$(-\Delta)^{\ln}$ on $\Omega$. By \cite[Theorem~1.4(iii)]{ChenWeth2019}, the first Dirichlet
eigenfunction is strictly positive in $\Omega$. Hence
\begin{equation}\label{eq:positive-set}
u_1>0 \quad\text{a.e. in }\Omega,
\qquad
|\{u_1>0\}|=|\Omega|.
\end{equation}
Since $u_1=0$ a.e. on $\Omega^c$ and ${\rm diam}\Omega<1$, the second term in
\eqref{eq:log-form} vanishes. Therefore
\begin{equation}\label{eq:small-form}
\cE_L(u_1,u_1)
=
\frac{c_N}{2}\mathcal{D}_K(u_1)+\rho_N\|u_1\|_2^2.
\end{equation}
By the isodiametric inequality, $\operatorname{diam}(\operatorname{ess\,supp}u_1^*)
\leq
\operatorname{diam}(\Omega)<1.$
Hence \eqref{eq:small-form} also applies to $u_1^*$, and
Proposition~\ref{prop:strict} yields
\[
\cE_L(u_1^*,u_1^*)
\leq
\cE_L(u_1,u_1)
=
\lambda_1^{\ln}(\Omega).
\]
Since $u_1^*$ is admissible for the first Dirichlet eigenvalue of the ball $B$
and $\|u_1^*\|_2=1$,
\[
\lambda_1^{\ln}(B)
\le
\cE_L(u_1^*,u_1^*)
\le
\lambda_1^{\ln}(\Omega).
\]

Assume now that $\lambda_1^{\ln}(\Omega)=\lambda_1^{\ln}(B).$
After the same scaling reduction, the preceding chain consists entirely of
equalities. In particular, $\mathcal{D}_K(u_1)=\mathcal{D}_K(u_1^*).$
Proposition~\ref{prop:strict} yields an $a\in\R^N$ such that
\[
 u_1(x)=u_1^*(x-a)
 \qquad\text{for a.e. }x.
\]
Using \eqref{eq:positive-set}, we conclude that
\[
\big|\big(\Omega\setminus (a+B)\big)\cup\big((a+B)\setminus \Omega \big) |=0.
\]
Because both $\Omega$ and $a+B$ are Lipschitz domains, equality up to a null
set implies equality as open sets. Hence $\Omega=a+B.$
Conversely, translation invariance of $(-\Delta)^{\ln}$ gives equality whenever
$\Omega$ is a translate of $B$. This completes the proof.
\end{proof}


\bigskip

{\small
\noindent\textbf{Conflict of interest.}
The authors declare that they have no conflict of interest.

\medskip
\noindent\textbf{Data availability.}
No data were used for the research described in this article.

\medskip
\noindent\textbf{Acknowledgements.}
H. Chen is supported by  NSFC, no. 12361043. \\
R. Chen is supported by China Scholarship Council,  Liujinxuan [2025] no. 37. \\
 B. Hua is supported by NSFC, no.12371056. 

\medskip
\noindent\textbf{AI assistance statement.}
The authors used OpenAI models as assistive tools in preparing this manuscript. All mathematical arguments,
proofs, and verifications were carried out by the authors, who take full
responsibility for the content of the paper.
}

\bibliographystyle{amsplain}
\bibliography{references}

\end{document}